\documentclass[12pt]{amsart}
\usepackage{color}
\usepackage{amssymb,pb-diagram}
\usepackage[?]{amsrefs}

\newtheorem{Thm}{Theorem}[section]

\newtheorem{Cor}[Thm]{Corollary}
\newtheorem{Lem}[Thm]{Lemma}
\newtheorem{Prop}[Thm]{Proposition}

\theoremstyle{definition}

\theoremstyle{remark}
\newtheorem{remark}{Remark}

\def\ldots{\mathinner{\ldotp\ldotp\ldotp}}
\def\cdots{\mathinner{\cdotp\cdotp\cdotp}}

\def \cal{\mathcal}

\def \diam{\text{diam }}

\def\eps{\varepsilon}

\def\Ndb{\mathbb N}

\newcommand{\brls}{\sum_{t\le s}x_{l+1,t}}
\newcommand{\brlls}{\sum_{t\le s}x_{l+2,t}}
\newcommand{\brksp}{\sum_{t\le s'}x_{k+1,t}}
\newcommand{\brkksp}{\sum_{t\le s'}x_{k+2,t}}

\def\Sz{\text{Sz }}

\begin{document}

\baselineskip=17pt

\title{A new metric invariant for Banach spaces}

\author{F. Baudier}

\address{Universit\'e de Franche-Comt\'e, Laboratoire de Math\'ematiques UMR 6623,
16 route de Gray, 25030 Besan\c con Cedex, FRANCE.}
\email{florent.baudier@univ-fcomte.fr}

\author{N. J. Kalton}
\address{Department of Mathematics \\
University of Missouri-Columbia \\
Columbia, MO 65211 }

\email{kaltonn@missouri.edu}

\author{G. Lancien}

\address{Universit\'e de Franche-Comt\'e, Laboratoire de Math\'ematiques UMR 6623,
16 route de Gray, 25030 Besan\c con Cedex, FRANCE.}
\email{gilles.lancien@univ-fcomte.fr}


\thanks{The first author acknowledges support from NSF grant DMS-0555670}

\begin{abstract}
We show that if the Szlenk index of a Banach space $X$ is larger than the first
infinite ordinal $\omega$ or if the Szlenk index of its dual is larger than
$\omega$, then the tree of all finite sequences of integers equipped with the
hyperbolic distance metrically embeds into $X$. We show that the converse is
true when $X$ is assumed to be reflexive. As an application, we exhibit new
classes of Banach spaces that are stable under coarse-Lipschitz embeddings and
therefore under uniform homeomorphisms.
\end{abstract}

\maketitle

\markboth{F. BAUDIER, N. J. KALTON AND G. LANCIEN}{A new metric invariant for
Banach spaces}

\section{Introduction}

In 1976 Ribe proved in \cite{Ribe1976} that two uniformly
homeomorphic Banach spaces are finitely representable in each other.
This theorem gave birth to the ``Ribe program'' (see
\cite{Bourgain1986} or \cite{MendelNaor2008} for a detailed
description). Local properties of Banach spaces are properties which
only involve finitely many vectors. These are properties which are
stable under finite representability. In view of Ribe's result the
``Ribe program'' aims at looking for metric invariants that
characterize local properties of Banach spaces. The first occurence
of the ``Ribe program'' is Bourgain's metric characterization of
superreflexivity given in \cite{Bourgain1986}. The metric invariant
discovered by Bourgain is the collection of the hyperbolic dyadic
trees of arbitrarily large height $N$. If we denote
$\Omega_{0}=\{\emptyset\}$, the root of the tree. Let
$\Omega_{i}=\{-1,1\}^{i}$, $B_{N}=\bigcup_{i=0}^{N}\Omega_{i}$. Thus
$B_N$ endowed with its shortest path metric $\rho$ is the hyperbolic
dyadic tree of height $N$.

Let us recall some definitions. Let $(M,d)$ and $(N,\delta)$ be two metric
spaces and let $f:M\to N$ be an injective map. The {\it distortion} of $f$ is
$$ {\rm dist}(f):= \|f\|_{Lip}\|f^{-1}\|_{Lip}=\sup_{x\neq y \in
M}\frac{\delta(f(x),f(y))}{d(x,y)}.\sup_{x\neq y \in
M}\frac{d(x,y)}{\delta(f(x),f(y))}.$$ If $\mbox{dist}(f)$ is finite, we say
that $f$ is a Lipschitz or metric embedding of $M$ into $N$. If there exists an
embedding $f$ from $M$ into $N$, with dist$(f)\leq C$, we use the notation $M
\buildrel {C}\over {\hookrightarrow} N$.

Bourgain's characterization is the following:

\begin{Thm} {\bf (Bourgain 1986)} Let $X$ be a Banach space. Then $X$ is not superreflexive if and only
if there exists a universal constant $C$ such that for all $N\in \Ndb$,
$(B_N,\rho)\buildrel {C}\over {\hookrightarrow} X$.
\end{Thm}

It has been proved in \cite{Baudier2007} that this is also equivalent to the
metric embedding of the infinite hyperbolic dyadic tree $(B_\infty,\rho)$ where
$B_\infty=\bigcup_{N=0}^{\infty}B_N$.\\

We also recall that it follows from the Enflo-Pisier renorming theorem
(\cite{Enflo1972} and \cite{Pisier1975}) that superreflexivity is equivalent to
the existence of an equivalent uniformly convex and (or) uniformly smooth norm.

In the series of papers \cite{BourgainMilmanWolfson1986},
\cite{MendelNaor2007}, \cite{MendelNaor2008} local properties such
as linear type and linear cotype are deeply studied and other
occurrences of ``Ribe's program'' are given.

In a similar vein our paper is an attempt to investigate which
asymptotic properties admit a metrical characterization. Asymptotic
properties have been intensively studied in
\cite{JohnsonLindenstraussPreissSchechtman2002},
\cite{GodefroyKaltonLancien2001} and \cite{OdellSchlumprecht2006}
and we refer to \cite{KnaustOdellSchlumprecht1999} for a precise
definition of the asymptotic structure of a Banach space. The main
result of this paper is an analogue of Bourgain's theorem in the
asymptotic setting. Let us first introduce a few notation and
definitions. For a positive integer $N$, We denote
$T_N=\bigcup_{i=0}^N \Ndb^i$, where $\Ndb^0:=\{\emptyset\}$. Then
$T_\infty=\bigcup_{N=1}^\infty T_N$ is the set of all finite
sequences of positive integers. For $s\in T_\infty$, we denote by
$|s|$ the length of $s$. There is a natural ordering on $T_\infty$
defined by $s\le t$ if $t$ extends $s$. If $s\le t$, we will say
that $s$ is an {\it ancestor} of $t$. If $s \le t$ and $|t|=|s|+1$,
we will say that $s$ is the {\it predecessor} of $t$ and $t$ is a
{\it successor} of $s$ and we will denote $s=t^-$. Then we equip
$T_\infty$, and by restriction every $T_N$, with the hyperbolic
distance $\rho$, which is defined as follows. Let $s$ and $s'$ be
two elements of $T_\infty$ and let $u\in T_\infty$ be their greatest
common ancestor. We set
$$\rho(s,s')=|s|+|s'|-2|u|=\rho(s,u)+\rho(s',u).$$

We now define the asymptotic version of uniform convexity and uniform
smoothness that we will consider. Let $(X,\|\ \|)$ be a Banach space and
$\tau>0$. We denote by $B_X$ its closed unit ball and by $S_X$ its unit sphere.
For $x\in S_X$ and $Y$ a closed linear subspace of $X$, we define
$$\overline{\rho}(\tau,x,Y)=\sup_{y\in S_Y}\|x+\tau y\|-1\ \ \ \ {\rm and}\ \ \
\ \overline{\delta}(\tau,x,Y)=\inf_{y\in S_Y}\|x+\tau y\|-1.$$ Then
$$\overline{\rho}(\tau)=\sup_{x\in S_X}\ \inf_{{\rm
dim}(X/Y)<\infty}\overline{\rho}(\tau,x,Y)\ \ \ \ {\rm and}\ \ \ \
\overline{\delta}(\tau)=\inf_{x\in S_X}\ \sup_{{\rm
dim}(X/Y)<\infty}\overline{\delta}(\tau,x,Y).$$ The norm $\|\ \|$ is said to be
{\it asymptotically uniformly smooth} if
$$\lim_{\tau \to 0}\frac{\overline{\rho}(\tau)}{\tau}=0.$$
It is said to be {\it asymptotically uniformly convex} if
$$\forall \tau>0\ \ \ \ \overline{\delta}(\tau)>0.$$
These moduli have been first introduced by Milman in \cite{Milman1971}.

We can now state the main result of our paper in a way that is clearly an
asymptotic analogue of Bourgain's theorem.

\begin{Thm}\label{main} Let $X$ be a reflexive Banach space. The following assertions are
equivalent.

(i) There exists $C\ge 1$ such that $T_\infty \buildrel {C}\over
{\hookrightarrow} X$.

(ii) There exists $C\ge 1$ such that for any $N$ in $\Ndb$, $T_N \buildrel
{C}\over {\hookrightarrow} X$.

(iii) $X$ does not admit any equivalent asymptotically uniformly smooth norm
\underline{or} $X$ does not admit any equivalent asymptotically uniformly
convex norm.
\end{Thm}

The main tool for our proof will be the so-called {\it Szlenk index}. We now
recall the definition of the Szlenk derivation and the Szlenk index that have
been first introduced in \cite{Szlenk1968} and used there to show that there is
no universal space for the class of separable reflexive Banach spaces. So
consider a real separable Banach space $X$ and $K$ a weak$^*$-compact subset of
$X^*$. For $\eps>0$ we let $\cal V$ be the set of all relatively weak$^*$-open
subsets $V$ of $K$ such that the norm diameter of $V$ is less than $\eps$ and
$s_{\eps}K=K\setminus \cup\{V:V\in\cal V\}.$ We define inductively
$s_{\eps}^{\alpha}K$ for any ordinal $\alpha$, by
$s^{\alpha+1}_{\eps}K=s_{\eps}(s_{\eps}^{\alpha}K)$ and
$s^{\alpha}_{\eps}K={\displaystyle \cap_{\beta<\alpha}}s_{\eps}^{\beta}K$ if
$\alpha$ is a limit ordinal.  Then we define $Sz(X,\eps)$ to be the least
ordinal $\alpha$ so that $s_{\eps}^{\alpha}B_{X^*}=\emptyset,$ if such an
ordinal exists. Otherwise we write $Sz(X,\eps)=\infty.$ The {\it Szlenk index}
of $X$ is finally defined by $Sz(X)=\sup_{\eps>0}Sz(X,\eps).$

\noindent We denote $\omega$ the first infinite ordinal and
$\omega_1$ the first uncountable ordinal. Note that the dual of a
separable Banach space $X$ is separable if and only if
$Sz(X)<\omega_1$ (this is a consequence of Baire's theorem on the
pointwise limit of sequences of continuous functions). We will
essentially deal with the condition $Sz(X)\le \omega$. The
weak$^*$-compactness of $B_{X^*}$ implies that this is equivalent to
the condition: $Sz(X,\eps)<\omega$, for all $\eps>0$. Besides, it
follows from a theorem of Knaust, Odell and Schlumprecht
(\cite{KnaustOdellSchlumprecht1999}) that a separable Banach space
admits an equivalent asymptotically uniformly smooth norm if and
only if $Sz(X)\le \omega$. Then it is easy to see that for a
reflexive Banach space the condition $Sz(X^*)\le \omega$ is
equivalent to the existence of an equivalent asymptotically
uniformly convex norm on $X$. Therefore condition $(iii)$ in Theorem
(\ref{main}) is equivalent to

$$(iv)\ \ Sz(X)>\omega\ \  {\rm or}\ \  Sz(X^*)>\omega.$$

\medskip With this information at hand, we shall almost forget the formulations
in terms of renormings and work essentially with the notion of the
Szlenk index of a Banach space.

In order to have a complete view of the analogy between our result and
Bourgain's theorem, it is worth noting at this point that the superreflexivity
can be similarly characterized by the behavior  of an ordinal index. For a
given weak$^*$-compact convex subset $C$ of $X^*$ and a given $\eps>0$, let us
denote $\cal S$ be the set of all relatively weak$^*$-open slices $S$ of $C$
such that the norm diameter of $S$ is less than $\eps$ and
$d_{\eps}C=C\setminus \cup\{S:S\in\cal S\}$. We then define inductively
$d_\eps^\alpha(C)$ for $\alpha$ ordinal as before and $Dz(X,\eps)$ to be the
least ordinal $\alpha$ so that $d_{\eps}^{\alpha}B_{X^*}=\emptyset,$ if such an
ordinal exists. Otherwise we write $Dz(X,\eps)=\infty.$ Finally, the {\it
weak$^*$-dentability index} of $X$ is  $Dz(X)=\sup_{\eps>0}Dz(X,\eps).$ Then it
follows from \cite{Lancien1995} (see also the survey \cite{Lancien2006}) that
the following conditions are equivalent:

(i) $X$ is super-reflexive.

(ii) $Dz(X)\le \omega$.

(iii) $Dz(X^*)\le \omega$.

\medskip
Let us now describe the organization of this article. In Section 2
we give the construction of several embeddings and finally prove
that $T_\infty$ Lipschitz-embeds into $X$, whenever $Sz(X)>\omega$
or $Sz(X^*)>\omega$. In Section 3 we show the converse statement in
the reflexive case. This will conclude the proof of Theorem
\ref{main}. In the last section we describe a few applications of
our result to the stability of certain classes of Banach spaces
under coarse-Lipschitz embeddings or uniform homeomorphisms. The
main consequence of our work is that the class of all separable
reflexive spaces $X$ so that $Sz(X)\le \omega$ and $Sz(X^*)\le
\omega$ is stable under coarse-Lipschitz embeddings. It seems also
interesting to us that a metric invariant (the embeddability of
$T_\infty$ in this case) is used to prove stability results, whereas
the metric invariant is often looked after, when the class is
already known to be stable.

\section{Construction of the embeddings}\label{embeddings}

Before to start, we need to introduce more notation concerning our
trees. For $s=(s_1,\dots,s_n)$ and $t=(t_1,\dots,t_m)$ in
$T_\infty$, we denote
$$s\frown t=(s_1,\dots,s_n,t_1,\dots,t_m)\ \ {\rm and\ also}\ \ \emptyset\frown
t=t\frown\emptyset=t.$$

\noindent For $t\in T_\infty$ and $k\le |t|$, we denote $t_{|_k}$ the ancestor
of $t$ of length $k$.

\noindent For $s\le t$ in $T_\infty$, we denote $[s,t]=\{u\in T_\infty,\ s\le
u\le t\}$.

\noindent For $N$ in $\Ndb$ and $T\subset T_N$, we say that a map $\Phi: T_N\to
T$ is a {\it tree isomorphism} if $\Phi(T_N)=T$, $\Phi(\emptyset)=\emptyset$
and for all $s\in T_{N-1}$ and $n\in \Ndb$ $\Phi(s\frown n)=\Phi(s)\frown
k_{s,n}$ with $k_{s,n}\in \Ndb$ and $k_{s,n}<k_{s,m}$ whenever $n<m$. A subset
$T$  of $T_N$ is called a {\it full subtree} of $T_N$ if there exists a tree
isomorphism from $T_N$ onto $T$ or equivalently if $\emptyset \in T$ and for
all $s\in T\cap T_{N-1}$, the set of successors of $s$ that also belong to $T$
is infinite.

\medskip
We now begin with a very simple lemma.
\begin{Lem}\label{glidinghump}
Let $(x_n^*)_{n=0}^\infty$ be a weak*-null sequence in $X^*$ such that $\Vert
x_n^*\Vert\ge 1$ for all $n$ in $\Ndb$ and let $F$ be a finite dimensional
subspace of $X^*$. Then there exists a sequence $(x_n)_n$ in $B_X$ such that
for all $y^*\in F$, $y^*(x_n)=0$ and ${\liminf}\ x_n^*(x_n)\ge \frac{1}{2}$
\end{Lem}

\begin{proof} It is a classical consequence of Mazur's technique for
constructing basic sequences (see for instance
\cite{LindenstraussTzafriri1977}), that $\liminf d(x_n^*,F)\ge \frac{1}{2}$.
Denote $E=\{x\in X\ \ \forall x^*\in F\ \ x^*(x)=0\}$ be the pre-orthogonal of
$F$. Since $F$ is finite dimensional, we have that $F=E^\perp$. Therefore, for
any $x^*\in X^*$, $d(x^*,F)=\|x^*_{|_{E}}\|_{E^*}$. This finishes the proof.
\end{proof}

Let now $X$ be a separable Banach space. It follows from the metrizability of
the weak$^*$ topology on $B_{X^*}$ that if Sz$(X,\eps)>\omega$ then, for all
$N\in \Ndb$ there exists $(y_s^*)_{s\in T_N}$ in $B_X^*$ such that for all
$s\in T_{N-1}$ and all $n\in \Ndb$, $\Vert y_{s\frown n}^* -y_s^* \Vert\ge
\eps/2:=\eps'$ and $y^*_{s\frown n}\stackrel{w*}{\to}y^*_s$.

\noindent It is an easy and well known fact that the map $\eps \mapsto {\rm
Sz}(X,\eps)$ is submultiplicative (see for instance \cite{Lancien 2006}). So,
if Sz$(X)>\omega$, then Sz$(X,\eps)>\omega$ for any $\eps\in (0,1)$. Therefore,
in the above choice of $(y^*_s)_{s\in T_N}$ we can take $\eps'=\frac{1}{3}$.

\noindent By considering $z^*_s=y^*_s-y^*_{s^-}$ for $s\neq \emptyset$,
$z^*_\emptyset=y^*_\emptyset$ and re-scaling, this is clearly equivalent to the
existence, for all $N\in \Ndb$ of $(z_s^*)_{s\in T_N}$ in $X^*$ so that
\begin{itemize}
\item $\forall s\in T_N\setminus \{\emptyset\}$,  $\Vert z_s^* \Vert\ge 1$,
\item $\forall s\in T_{N-1}$, $z^*_{s\frown n}\stackrel{w*}{\to}0$,
\item $\forall s\in T_N, \Vert \sum_{t\le s}z_t^*\Vert\le 3$.
\end{itemize}

In our next proposition, we improve the above statement by constructing an
almost biorthogonal system associated with $(z_s^*)_{s\in T_N}$.

\begin{Prop}\label{biorthogonal}
Let $X$ be a separable Banach space. If Sz$(X)>\omega$, then for all $N\in
\Ndb$ and $\delta>0$ there exist $(x_s^*)_{s\in T_N}$ in $X^*$ and $(x_s)_{s\in
T_N}$ in $B_X$ such that
\begin{itemize}
 \item $\forall s\in T_{N-1},\ \ x^*_{s\frown n}\stackrel{w*}{\to}0$,
 \item $\forall s\in T_N\setminus\{\emptyset\},\ \ \Vert x_s^*\Vert\ge 1$
 and $\forall s\in T_N,\ \ \Vert \sum_{t\le s}x_t^*\Vert\le 3$,
 \item $\forall s\in T_N,\ \ x_s^*(x_s)\ge\frac{1}{3}\|x^*_s\|$,
 \item $\forall s\neq t,\ \ |x_s^*(x_t)|<\delta$.
\end{itemize}

\end{Prop}

\begin{proof} Let $f:\Ndb \to T_N$ be a bijection such that

$$\forall s<t \in T_N\ \ f^{-1}(s)<f^{-1}(t)$$ and
$$\forall s\in T_{N-1}\ \ \forall n<m\in \Ndb\ \  f^{-1}(s\frown n)<f^{-1}(s\frown
m).$$ Denote $s_i=f(i)$. In particular, $\emptyset=s_1$.

\noindent We now build inductively a tree isomorphism $\Phi: T_N\to
\Phi(T_N)\subset T_N$ and a family $(z_{\Phi(s)})_{s\in T_N}$ in $B_{X}$ such
that

\begin{equation}\label{biortho} z^*_{\Phi(s)}(z_{\Phi(s)})\ge \frac{1}{3},\ s \in T_N\  {\rm
and}\ |z^*_{\Phi(s)}(z_{\Phi(t)})|<\delta,\  s\neq t \in T_N.
\end{equation}
So set $\Phi(\emptyset)=\emptyset$, pick $z_{\Phi(\emptyset)}$ in
$B_X$ so that $z^*_{\Phi(\emptyset)}(z_{\Phi(\emptyset)})\ge
\frac{1}{3}\|z^*_{\Phi(\emptyset)}\|$ and assume that
$\Phi(s_1),\dots,\Phi(s_k)$ and $z_{\Phi(s_1)},\dots,z_{\Phi(s_k)}$
have been constructed accordingdly to (\ref{biortho}). Then, there
exists $i\in \{1,\dots,k\}$ and $p\in \Ndb$ such that
$s_{k+1}=s_i\frown p$. Since $(z^*_{\Phi(s_i)\frown n})_{n\ge 1}$ is
a weak$^*$-null sequence, Lemma \ref{glidinghump} insures that we
can pick $n\in \Ndb$ and $z_{\Phi(s_i)\frown n}$ in $B_X$ such that
$|z^*_{\Phi(s_i)\frown n}(z_{\Phi(s_j)})|<\delta$ for all $j\le k$,
$z^*_{\Phi(s_j)}(z_{\Phi(s_i)\frown n})=0$ for all $j\le k$ and
$z^*_{\Phi(s_i)\frown n}(z_{\Phi(s_i)\frown n})\ge \frac{1}{3}$. We
now set $\Phi(s_{k+1})=\Phi(s_i)\frown n$. If $n$ is chosen large
enough all the required properties, including those needed for
making $\Phi$ a tree isomorphism, are satisfied.

\noindent We conclude the proof by setting $x^*_s=z^*_{\Phi(s)}$ and
$x_s=z_{\Phi(s)}$, for $s$ in $T_N$.

\end{proof}

We shall improve progressively our embedding results and start with the
following.

\begin{Prop}\label{finitetrees} There is a universal constant $C\ge 1$ such that, whenever
$X$ is a separable Banach space with Sz$(X)>\omega$, we have that
$$\forall N\in \Ndb\ \ \ T_N\buildrel {C}\over {\hookrightarrow}X \ \
{\sl and}\ \ T_N\buildrel {C}\over {\hookrightarrow} X^*.$$
\end{Prop}

\begin{proof} Let $(x^*_s,x_s)_{s\in T_N}$ be the system given by Proposition
\ref{biorthogonal}. Our choice of $\delta$, will be specified later.

\medskip
We shall first embed the $T_N$'s into $X$. For that purpose, we mimic the
natural embedding of $T_N$ into $\ell_1(T_N)$ (with $(x_t)_{t\in T_N}$ playing
the role of the canonical basis of $\ell_1(T_N)$) and define $F: T_N\to X$ by
$$\forall s\in T_N\ \ F(s)=\sum_{t\le s} x_t.$$
Since $(x_t)_{t\in T_N} \subset B_X$, we clearly have that $F$ is 1-Lipschitz
for the metric $\rho$ on $T_N$.

\noindent Let now $s\neq s'$ in $T_N$ and let $u$ be their greatest common
ancestor. Denote $d=\rho(u,s)$ and $d'=\rho(u,s')$. Recall that
$\rho(s,s')=d+d'$ and assume for instance that $d\ge d'$. Then
$$\langle \sum_{t\le s}x^*_t,F(s)-F(s')\rangle \ge \frac{1}{3}d-\delta |s|(d+d')\ge
\frac{d}{3}-2N^2\delta\ge \frac{1}{4}d\ge \frac{1}{8}\rho(s,s'),$$ if $\delta$
was chosen less than $\frac{1}{24N^2}$.

\noindent Since $\|\sum_{t\le s}x^*_t\|\le 3$, we obtain that for all $s,s'$ in
$T_N$:
$$\|F(s)-F(s')\|\ge\frac
{1}{24}\rho(s,s').$$ This finishes the  proof of our first embedding result.

\medskip We now turn to the question of embedding the $T_N$'s into $X^*$.

\noindent Our construction will copy the natural embedding of $T_N$ into
$c_0(T_N)$, with $(x^*_t)_{t\in T_N}$ replacing the canonical basis of
$c_0(T_N)$. For $s\in T_N$, we denote $y^*_s=\sum_{t\le s} x^*_t$. Then we
define $G:T_N \to X^*$ by
$$\forall s\in T_N\ \ G(s)=\sum_{t\le s} y_t^*.$$
Since $(y^*_t)_{t\in T_N}$ is a subset of $3B_{X^*}$, it is immediate that $G$
is 3-Lipschitz.

\noindent Let now $s\neq s'$ in $T_N$ and denote again $u$ their greatest
common ancestor, $d=\rho(u,s)$ and $d'=\rho(u,s')$. Assume for instance that
$d\ge d'$. Let us name $v$ the unique successor of $u$ such that $v\le s$ and
$w$ the unique successor of $u$ such that $w\le s'$ if it exists. Then
$$G(s)-G(s')=\sum_{v\le t\le
s}y^*_t -\sum_{w\le t\le s'}y^*_t.$$ If $s'\le s$, $[w,s']$ is empty.
Otherwise,
$$\forall t\in [w,s']\ \ \ |\langle x_v,y^*_t\rangle| \le \delta |t|\le \delta
N.$$ On the other hand
$$\forall t\in [v,s]\ \ \ |\langle x_v,y^*_t\rangle| \ge \frac{1}{3}-\delta (|t|-1)\ge \frac{1}{3} -\delta
N.$$ The two previous inequalities yield
$$\|G(s)-G(s')\|\ge |\langle x_v,G(s)-G(s')\rangle|\ge \frac{d}{3}-2\delta N^2\ge
\frac{d}{4}\ge \frac{1}{8}\rho(s,s'),$$ if $\delta$ was chosen in
$(0,\frac{1}{24 N^2})$. This concludes our argument for the second embedding.

\end{proof}

\begin{remark}Let us just finally notice that in both cases we proved the statement for
$C=24$, but our argument allows us to get the result for any constant $C>8$.
\end{remark}

\begin{remark} The end of this section will be devoted to various improvements of
Proposition \ref{finitetrees}, which are not fully needed in order
to read the last two sections.
\end{remark}

\medskip
We now turn to the problem of embedding $T_\infty$. We shall refine our
arguments in order to improve Proposition \ref{finitetrees} and obtain:

\begin{Thm}\label{infinitetree} There is a constant $C\ge 1$ such that for any separable Banach space
$X$ satisfying Sz$(X)>\omega$, we have $$T_\infty \buildrel {C}\over
{\hookrightarrow} X \ \ {\rm and}\ \ T_\infty \buildrel {C}\over
{\hookrightarrow} X^*.$$
\end{Thm}
Although this statement implies our previous results, we have chosen to
separate its proof in the hope of making it easier to read.

\begin{proof} So assume that $Sz(X)>\omega$ and fix a decreasing sequence
$(\delta_i)_{i=0}^\infty$ in $(0,1)$. By combining the technique of Proposition
\ref{biorthogonal} and a proper enumeration of $\bigcup_{i=0}^\infty
\{i\}\times T_{2^i}$, one can actually build for every $i\ge 0$:
$(x_{i,s}^*)_{s\in T_{2^i}}$ in $X^*$ and $(x_{i,s})_{s\in T_{2^i}}$ in $B_X$
such that

(i) $\forall i\ge 0,\ \forall s\in T_{2^i-1},\ \ x^*_{i,s\frown
n}\stackrel{w*}{\to}0$,

(ii) $\forall i\ge 0,\ \forall s\in T_{2^i}\setminus\{\emptyset\},\ \ \Vert
x_{i,s}^*\Vert\ge 1$
 and \ $\forall s\in T_{2^i},\ \Vert \sum_{t\le s}x_{i,t}^*\Vert\le 3$,

(iii) $\forall i\ge 0,\ \forall s\in T_{2^i},\ \
x_{i,s}^*(x_{i,s})\ge\frac{1}{3}\|x^*_{i,s}\|$,

(iv) $\forall (i,s)\neq (j,t),\ \ |x_{i,s}^*(x_{j,t})|<\delta_i$.

\noindent Let us just emphasize the fact that the whole system
$(x_{i,s},x^*_{i,s})_{(i,s)}$ is almost biorthogonal. We wish also to note that
the estimate given in (iv) depends only on $i$. This last fact relies on a
careful application on Lemma \ref{glidinghump}.

\noindent For $i\ge 0$, we denote $F_i$ a translate of the map defined on
$T_{2^{i+1}}$ in the proof of Proposition \ref{finitetrees}. So let
$$F_i(\emptyset)=0\ \ {\rm and}\ \ F_i(s)=\sum_{\emptyset <t\le s} x_{i+1,t}\ \ \
s\in T_{2^{i+1}}\setminus\{\emptyset\}.$$ Now we adopt the gluing technique
introduced in \cite{Baudier2007} and also used in \cite{BaudierLancien2008} and
build our embedding as follows. For $s\in T_\infty\setminus\{\emptyset\}$ there
exists $k\ge 0$ such that $2^k\le |s|< 2^{k+1}$. We define
$$F(s)=\lambda_s
F_k(s)+(1-\lambda_s)F_{k+1}(s),\ \ {\rm where}\ \
\lambda_s=\frac{2^{k+1}-|s|}{2^{k}}.$$ Of course, we set $F(\emptyset)=0$. We
clearly have that for all $s\in T_\infty$, $\|F(s)\|\le |s|$ and following the
proof of Theorem 2.1 in \cite{BaudierLancien2008} that $F$ is 9-Lipschitz.
Consider now $s\neq s'\in T_\infty\setminus\{\emptyset\}$ and assume for
instance that $1\leq |s'|\leq |s|$. Let $2^k\le \vert s'\vert\le 2^{k+1}$ and
$2^l\le \vert s\vert\le 2^{l+1}$, with $k\le l$. Then,
\begin{equation*}
\begin{split}
F(s)-F(s')=&\lambda_s\brls+(1-\lambda_s)\brlls\\
&-\left(\lambda_{s'}\brksp+(1-\lambda_{s'})\brkksp\right).
\end{split}
\end{equation*}
Let $u$ be the greatest common ancestor of $s$ and $s'$.\\ If we denote
$(\ast)=\langle \displaystyle\sum_{u<t\le
s}(x_{l+1,t}^*+x_{l+2,t}^*),F(s)-F(s')\rangle$, we get

\begin{equation*}
\begin{split}
(\ast) & \ge \lambda_s\frac{d}{3}+(1-\lambda_s)\frac{d}{3}\\
& -\delta_{l+1}(\lambda_s d(\vert s\vert-1)+(1-\lambda_s)d\vert
s\vert+
\lambda_{s'}d\vert s'\vert+(1-\lambda_{s'})d\vert s'\vert)\\
& -\delta_{l+2}((1-\lambda_s) d(\vert s\vert-1)+\lambda_s d\vert
s\vert+
                        \lambda_{s'}d\vert s'\vert+(1-\lambda_{s'})d\vert s'\vert)\\
& \ge \frac{d}{3}-2d\vert s\vert(\delta_{l+1}+\delta_{l+2})\\
& \ge \frac{d}{3}-2\cdotp2^{2l+2}(\delta_{l+1}+\delta_{l+2})\ge
\frac{d}{4}\ge \frac{\rho(s,s')}{8},
\end{split}
\end{equation*}
if the $\delta_i$'s were chosen small enough.

\noindent Since $\|\sum_{t\le s} x^*_{i,t}\|\le 3$ \ for all $i\ge 0$, we
obtain the following lower bound
$$\Vert F(s)-F(s')\Vert\ge \frac{\rho(s,s')}{96}.$$
If $s'=\emptyset\neq s'$, the argument is similar but simpler. This concludes
our proof.

\medskip In order to embed $T_\infty$ into $X^*$, we use exactly the same
technique. For $i\ge 0$ and $s\in T_{2^i}$ denote $y^*_{i,s}=\sum_{t\le s}
x^*_{i,t}$ and
$$G_i(\emptyset)=0\ \ {\rm and}\ \ G_i(s)=\sum_{\emptyset <t\le s} y^*_{i+1,t},\ \
s\in T_{2^{i+1}}\setminus\{\emptyset\}.$$ Then again, we set $G(\emptyset)=0$
and for $s\in T_\infty\setminus\{\emptyset\}$:
$$G(s)=\lambda_s
G_k(s)+(1-\lambda_s)G_{k+1}(s).$$ Following again the proof in
\cite{BaudierLancien2008}, we obtain first that $G$ is 27-Lipschitz.

\noindent Consider now $s\neq s'\in T_\infty$ such that for instance $0\leq
|s'|\leq |s|$, $2^l\le \vert s\vert\le 2^{l+1}$ and  $2^k\le \vert s'\vert\le
2^{k+1}$ with $k\le l$ or $s'=\emptyset$. Let $u$ be the greatest common
ancestor of $s$ and $s'$ and $v$ be the successor of $u$ such that $v\le s$. In
a very similar way, by evaluating $\langle x_{l+1,v}+x_{l+2,v},
G(s)-G(s')\rangle$, we can show that a proper choice for the $\delta_i$'s
implies that
$$\|G(s)-G(s')\|\ge \frac{\rho(s,s')}{16}.$$
This concludes the proof of this proposition.

\end{proof}

We will now study the condition ``Sz$(X^*)>\omega$''. We already know that if
Sz$(X^*)>\omega$, then $T_\infty$ Lipschitz embeds into $X^{**}$ and therefore,
when $X$ is reflexive, $T_\infty$ Lipschitz embeds into $X$. We will show how
to drop the reflexivity assumption in this statement. As before, we start with
finite trees.

\begin{Prop}\label{c0likefinite} There is a universal constant $C\ge 1$ such that, whenever
$X$ is a separable Banach space with Sz$(X^*)>\omega$, we have that
$$\forall N \in \Ndb,\ \ \  T_N\buildrel {C}\over {\hookrightarrow} X.$$
\end{Prop}

\begin{proof} If $X^*$ is non separable, then Sz$(X)>\omega$ and our problem is settled by Proposition
\ref{finitetrees}. Thus we assume that $X^*$ is separable. Then, for a given
positive integer $N$ and a given $\delta>0$, Proposition \ref{biorthogonal}
provides us with $(x^*_s)_{s\in T_N}$ in $B_{X^*}$ and $(x_s^{**})_{s\in T_N}$
in $X^{**}$ such that

$\bullet\ \forall s\in T_{N-1},\ \ x^{**}_{s\frown
n}\stackrel{w*}{\to}0$,

$\bullet\ \forall s\in T_N\setminus\{\emptyset\},\ \ \Vert
x_s^{**}\Vert\ge 1$ and \ $\forall s\in T_N,\ \Vert \sum_{t\le
s}x_t^{**}\Vert\le 3$,

$\bullet\ \forall s\in T_N,\ \
x_s^{**}(x^*_s)\ge\frac{1}{3}\|x^{**}_s\|$,

$\bullet\ \forall s\neq t,\ \ |x_s^{**}(x^*_t)|<\delta$.

\noindent Let $\{s_i,\ i\in \Ndb\}$ be an enumeration of $\{s\in T_N,\ |s|=N\}$
and let $\cal B_i=\{t\in T_N,\ t\le s_i\}$ be the corresponding branches of
$T_N$.

For $s\in T_N$ denote $y^{**}_s=\sum_{t\le s} x^{**}_t$.

\noindent Let us now fix $\eta>0$. For a given $s\in T_N$, there is
a unique $i=i_s\in \Ndb$ such that $s\in \cal B_{i_s} \setminus \cal
B_{i_s-1}$. Then, we can pick $y_s$ in $X$ so that

\begin{equation}\label{choiceofy}\|y_s\|\le 3\ \ \ {\rm and}\ \ \ \forall t\in
\bigcup_{j=1}^{i_s}\cal B_j\ \ |\langle
x^*_t,y^{**}_s-y_s\rangle|<\eta.
\end{equation}
In particular
\begin{equation}\label{choiceofy2}
\forall t\le s\ \ \ |\langle x^*_t,y^{**}_s-y_s\rangle|<\eta.
\end{equation}
We now define $G:T_N \to X$ by
$$\forall s\in T_N\ \ G(s)=\sum_{t\le s} y_t.$$
Since $(y_t)_{t\in T_N}$ is a subset of $3B_X$, it is immediate that $F$ is
3-Lipschitz.

\noindent Let now $s\neq s'$ in $T_N$ and denote again $u$ their greatest
common ancestor, $d=\rho(u,s)$ and $d'=\rho(u,s')$, $v$ the successor of $u$ so
that $v\le s$ and $w$ the successor of $u$ so that $w\le s'$, if they exist.

\noindent Assume first that $s$ and $s'$ are comparable and for
instance that $s'\le s$. Then $u=s'$, $v$ exists, $w$ does not and
by (\ref{choiceofy2})

\begin{equation}\label{evaluate} \langle x^*_v, F(s)-F(s')\rangle \ge \langle x^*_v,\sum_{v\le t\le
s}y^{**}_t\rangle -\eta d \ge \frac{1}{4}d,
\end{equation}
for $\delta$ and $\eta$ chosen small enough.

\noindent Suppose now that $s$ and $s'$ are not comparable. Then $v$ and $w$
are defined and not comparable. Therefore $i_v\neq i_w$. For instance
$i_v<i_w$. We will then consider two cases.

\noindent (a) If $d'\ge 24d$. Then $\|G(s)-G(s')\|\ge \|\sum_{u<t\le
s'}y_t\|-3d$. From (\ref{evaluate}) it follows that
$$\|G(s)-G(s')\|\ge \frac{1}{4} d'-3d\ge \frac{1}{8}d'\ge
\frac{1}{16}\rho(s,s').$$

\noindent (b) Assume now that $d'<24 d$.

\noindent We clearly have that for all $t$ in $[v,s]\cup [w,s']$,
$i_t\ge i_v$ and therefore it follows from (\ref{choiceofy})
$$\forall
t \in [v,s]\cup [w,s'],\ \ |\langle
x^*_v,y^{**}_t-y_t\rangle|<\eta.$$ It follows that
$$\|G(s)-G(s')\|\ge \langle x^*_v,\sum_{v\le t\le
s}y^{**}_t -\sum_{w\le t\le s'}y^{**}_t\rangle - (d+d')\eta \ge \frac{1}{4}d\ge
\frac{1}{100}\rho(s,s'),$$ if $\delta$ and $\eta$ were beforehand carefully
chosen small enough.

\end{proof}

We now state the last result of this section.

\begin{Thm}\label{c0likeinfinite} There is a universal constant $C\ge 1$ such that, whenever
$X$ is a separable Banach space with Sz$(X^*)>\omega$, we have that
$T_\infty\buildrel {C}\over {\hookrightarrow} X.$
\end{Thm}

\begin{proof} Again, we may directly assume that $X^*$ is separable. The gluing
argument that we used before to embed $T_\infty$ does not seem to be efficient
in this case. We shall develop another technique. Fix first an integer $K\ge
2$. Then choose a decreasing sequence $(\delta_i)_i$ in $(0,1)$. Assuming that
Sz$(X^*)>\omega$ we can build

$(x_{i,s}^{**})_{s\in T_{K^i+1}}$ in $X^{**}$ and $(x^*_{i,s})_{s\in
T_{K^i+1}}$ in $B_{X^*}$ such that

$\bullet\ \forall i\ge 0,\ \forall s\in T_{K^i},\ \
x^{**}_{i,s\frown n}\stackrel{w*}{\to}0$,

$\bullet\ \forall i\ge 0,\ \forall s\in
T_{K^i+1}\setminus\{\emptyset\},\ \ \Vert x_{i,s}^{**}\Vert\ge 1$
 and \ $\forall s\in T_{K^i+1},\ \Vert \sum_{t\le s}x_{i,t}^{**}\Vert\le 3$,

$\bullet\ \forall i\ge 0,\ \forall s\in T_{K^i+1},\ \
x_{i,s}^{**}(x^*_{i,s})\ge\frac{1}{3}\|x^{**}_{i,s}\|$,

$\bullet\ \forall (i,s)\neq (j,t),\ \
|x_{i,s}^{**}(x^*_{j,t})|<\delta_i$.

For $s$ in $T_{K^i+1}$, we define $y_{i,s}^{**}=\sum_{t\le s}
x_{i,s}^{**}$.

\noindent Let $N_i=\sum_{k=0}^i K^k$, choose an enumeration
$\{s_r^i,\ r\in \Ndb\}$ of $\{s\in T_{N_i},\ |s|=N_i\}$ and denote
$\cal B_r^i=\{t\in T_{N_i}, t\le s_r^i\}$ the branch of $T_{N_i}$
whose endpoint is $s_r^i$. We will also use an enumeration
$\{t_r^i,\ r\in \Ndb\}$ of the terminal nodes of $T_{K^i+1}$ and the
corresponding branches $\cal C_r^i=\{t\in T_{K^i+1},\ t\le t_r^i\}$.

Let us first describe the general idea. We set $G(\emptyset)=0$.
Consider now $s\in T_\infty\setminus \{\emptyset\}$. Then, there
exists $n\in \Ndb$ and $s_0,\dots,s_n$ in $T_\infty$ such that
$|s_j|=K^j$ for $j\le n-1$, $1\le |s_n|\le K^n$ and
$s=s_0\frown\dots\frown s_n$. For $j\le n-1$,
$s_0\frown\dots\frown s_j$ is a terminal node of $T_{N_j}$
that we denote $s_{r_j}^j$. We shall now define
$$G(s)=\sum_{\emptyset<t\le s_0}y_{t,0}+\dots+\sum_{r_{j-1}\le t\le r_{j-1}\frown
s_j} y_{t,j}+\dots+\sum_{r_{n-1}\le t\le r_{n-1}\frown s_n}
y_{t,n}$$ where $y_{t,j}$ is a proper weak$^*$-approximation of
$y_{t,j}^{**}$.

We now detail the rather technical construction of the $y_{t,j}$'s.

\noindent So let $s=(s(1),\dots,s(k))\in T_{K^i+1}\setminus
\{\emptyset\}$. We recall that $s_{s(1)}^{i-1}$ is the $s(1)^{\rm
th}$ terminal node of $T_{N_{i-1}}$. So it can be written
$s_{s(1)}^{i-1}=s_0\frown\dots\frown s_{i-1}$, with
$|s_j|=K^j$ for $j\le i-1$. Then, for any $j\le i-1$,
$s_0\frown\dots\frown s_j$ is a terminal node of $T_{N_j}$
that we denote $s_{r_j}^j$. Besides, $r_{j-1}\frown s_j$ is a
terminal node of $T_{K^j+1}$ that we denote $t_{k_j}^j$. Let also
$k_i\in\Ndb$ be such that $s\in \cal C_{k_i}^i\setminus
\bigcup_{k=1}^{k_i-1} \cal C_k^i$. Then we pick $y_{s,i}$ in $3B_X$
satisfying the following conditions:

\begin{equation}\label{approx}\forall j\le i\ \  \forall t\in \bigcup_{k=1}^{k_j} \cal C_k^j\ \ \ |\langle
y^{**}_{s,i}-y_{s,i},x^*_{t,j}\rangle|\le \delta_i
\end{equation}

Since any $y_{s,i}$ belongs to $3B_X$, it is clear that $G$ is
3-Lipschitz.

\medskip
We now start a discussion to prove that $G^{-1}$ is Lipschitz. So
let $s\neq s'$ in $T_\infty\setminus \{\emptyset\}$ and $n,m$ non
negative integers so that $N_{n-1}< |s|\le N_n$ and $N_{m-1}<|s'|\le
N_m$ (with the convention $N_{-1}:=0$). As usual, $u$ is the
greatest common ancestor of $s$ and $s'$ and we denote $p$ the
integer such that $N_{p-1}<|u|\le N_p$, $d=\rho(u,s)$ and
$d'=\rho(u,s')$. So we can write $s=s_0\frown\dots\frown s_n$,
$s'=s'_0\frown\dots\frown s'_m$ and
$u=u_0\frown\dots\frown u_p$, with $|s_j|=K^j$ for $j\le n-1$,
$0<|s_n|\le K^n$, $|s'_j|=K^j$ for $j\le m-1$, $0<|s'_m|\le K^m$,
$|u_j|=K^j$ for $j\le p-1$ and $0<|u_p|\le K^p$. Then, we have that
$u_j=s_j=s'_j$ for $j\le p-1$ and that $u_p$ is the greatest common
ancestor of $s_p$ and $s'_p$ in $T_{K^p}$. Finally, if we denote
$s_0\frown\dots\frown s_j=s_{r_j}^j$ for $j\le n-1$ and
$s'_0\frown\dots\frown s'_j=s^j_{r'_j}$ for $j\le m-1$, we can
write
$$G(s)-G(s')=\sum_{r_{p-1}\frown u_p < t\le r_{p-1}\frown
s_p} y_{t,p}+\dots+\sum_{r_{n-1}\le t\le r_{n-1}\frown s_n}
y_{t,n}$$
$$-\sum_{r_{p-1}\frown u_p < t\le r_{p-1}\frown s'_p}
y_{t,p}-\dots-\sum_{r'_{m-1}\le t\le r'_{m-1}\frown s'_m}
y_{t,m}.
$$

\medskip {\bf a)} Assume first that $n\ge m+2$.

\noindent Denote $x^*=x^*_{r_{n-2},n-1}$. Then

\begin{equation*}
\begin{split}
&\|G(s)-G(s')\|\\
&\ge \langle x^*,\sum_{r_{n-2}\le t\le r_{n-2}\frown s_{n-1}}
y_{t,n-1} + \sum_{r_{n-1}\le t\le r_{n-1}\frown s_n} y_{t,n} \rangle
-
6N_{n-2}\\
&\ge \langle x^*,\sum_{r_{n-2}\le t\le r_{n-2}\frown s_{n-1}}
y^{**}_{t,n-1} + \sum_{r_{n-1}\le t\le r_{n-1}\frown s_n}
y^{**}_{t,n} \rangle -
\delta_{n-1}(K^{n-1}+K^n)- 6N_{n-2}\\
&\ge
\frac{1}{3}K^{n-1}-\delta_{n-1}\big(K^{n-1}+K^n+K^{n-1}K^{n-1}+K^n(K^n+1)\big)-6N_{n-2}
\ge \frac{1}{4}K^{n-1},
\end{split}
\end{equation*}
if $K$ was chosen big enough and the $\delta_n$'s small enough.

\noindent In that case $\rho(s,s')\le 2N_n$. So
$$\|G(s)-G(s')\|\ge \frac{\rho(s,s')}{L},$$ where $L$ is  a constant depending
only on $K$.

\medskip {\bf b)} Assume that $n=m+1$ and $m=p$.

\noindent Denote $x^*=x^*_{r_{n-1},n}$, $a=|s_n|$ and
$b=K^{n-1}-|u_{n-1}|$. Notice that $a+b=d$ and $d'\le b$. Then

\begin{equation}\label{eq1} \langle x^*,G(s)-G(s') \rangle \ge \frac{a}{4}- 3b-3d'\ge
\frac{a}{4}-6b,
\end{equation}
if the $\delta_n$'s were chosen small enough.

\noindent Let $v_{n-1}$ be the successor of $u_{n-1}$ so that
$v_{n-1}\le s_{n-1}$ and $w_{n-1}$ be the successor of $u_{n-1}$ so
that $w_{n-1}\le s'_{n-1}$. Denote now $y^*=x^*_{r_{n-2}\frown
v_{n-1}}$ and $z^*=x^*_{r_{n-2}\frown w_{n-1}}$.

\noindent Assume first that there exists an integer $k$ such that
$r_{n-2}\frown v_{n-1}\in \cal C_k^{n-1}$ and $r_{n-2}\frown
w_{n-1}\notin \bigcup_{l=1}^k \cal C_l^{n-1}$. Then, for small
enough $\delta_n$'s

\begin{equation}\label{eq2}
\langle y^*,G(s)-G(s') \rangle \ge \frac{b}{4}.
\end{equation}
It follows from (\ref{eq1}) and (\ref{eq2}) that
$$\langle x^*+ 25y^*,G(s)-G(s') \rangle \ge \frac{a+b}{4}\ge \frac{d+d'}{8}.$$
Thus $$\|G(s)-G(s')\|\ge \frac{\rho(s,s')}{208}.$$

\noindent Assume now that there exists an integer $k$ such that
$r_{n-2}\frown w_{n-1}\in \cal C_k^{n-1}$ and $r_{n-2}\frown
v_{n-1}\notin \bigcup_{l=1}^k \cal C_l^{n-1}$. Then, still for small
$\delta_n$'s,
$$\langle y^*,G(s)-G(s') \rangle \ge
\frac{b}{4}-3d'\ \ {\rm and}\ \ \langle z^*,G(s)-G(s') \rangle \ge
\frac{d'}{4}.$$ It follows from the above and (\ref{eq1}) that
$$\langle x^*+25 y^*+301 z^*,G(s)-G(s') \rangle \ge
\frac{d+d'}{4}\ \ {\rm and} \  \ \|G(s)-G(s')\|\ge
\frac{\rho(s,s')}{1308}.$$

\medskip {\bf c)} Assume that $n=m+1$ and $p\le m-1$.

\noindent Denote $x^*=x^*_{r_{n-1},n}$, $y^*=x^*_{r_{n-2},n-1}$ and
$z^*=x^*_{r'_{n-2},n-1}$. Note that $y^*\neq z^*$. We also denote
$a=|s_n|,\ b=|s_{n-1}|=K^{n-1}, \ b'=|s'_{n-1}|$ and

\noindent $c=|s_0\frown\dots\frown
s_{n-2}|-|u|=|s'_0\frown\dots\frown s'_{n-2}|-|u|$.

\noindent First, we have that for small enough $\delta_n$'s

\begin{equation}\label{eq3}
\langle x^*,G(s)-G(s') \rangle \ge \frac{a}{4}-3b-3b'-6c\ge
\frac{a}{4} -6b-6c.
\end{equation}
Assume first that  there exists an integer $k$ such that $r_{n-2}\in
\cal C_k^{n-1}$ and $r'_{n-2}\notin \bigcup_{l=1}^k \cal C_l^{n-1}$.
Then, for small enough $\delta_n$'s
$$\langle y^*,G(s)-G(s') \rangle \ge \frac{b}{4}-6c.$$
This, together with (\ref{eq3}) yields
$$\langle x^*+25 y^*,G(s)-G(s') \rangle \ge \frac{a+b}{4}-156c.$$
A previous choice of a big enough $K$ insures in this situation that
$$\frac{a+b}{4}-156 c\ge \frac{\rho(s,s')}{10}-156 c\ge
\frac{\rho(s,s')}{20}.$$ Therefore
$$\|G(s)-G(s')\|\ge \frac{\rho(s,s')}{520}.$$
Otherwise, there exists an integer $k$ such that $r'_{n-2}\in \cal
C_k^{n-1}$ and $r_{n-2}\notin \bigcup_{l=1}^k \cal C_l^{n-1}$. Then
a proper choice for the $\delta_n$'s yields
\begin{equation}\label{eq4}
\langle y^*,G(s)-G(s') \rangle \ge \frac{b}{4}-3b'-6c\ \ {\rm and}\
\ \langle z^*,G(s)-G(s') \rangle \ge \frac{b'}{4}-6c
\end{equation}
From (\ref{eq3}) and (\ref{eq4}) we deduce
$$\langle x^*+ 25y^*+300z^*,G(s)-G(s') \rangle \ge \frac{a+b}{4}-1956c.$$
Again, our starting choice of a very large $K$ will insure the
existence of a universal constant $L$ so that in this situation
$$\|G(s)-G(s')\|\ge \frac{\rho(s,s')}{L}.$$

\medskip {\bf d)} Assume that $n=m=p$. We just have to follow
the proof of Proposition \ref{c0likefinite}

\medskip {\bf e)} Assume that $n=m$ and $p\le n-2$.

\noindent Denote $y^*=x^*_{r_{n-2},n-1}$, $z^*=x^*_{r'_{n-2},n-1}$,
$a=|s_n|$, $a'=|s'_n|$, $b=|s_{n-1}|=|s'_{n-1}|=K^{n-1}$ and
$c=|s_0\frown\dots\frown
s_{n-2}|-|u|=|s'_0\frown\dots\frown s'_{n-2}|-|u|$.

\noindent It follows from the condition (\ref{approx}) and a proper
choice of the $\delta_n$'s that
$${\rm either} \ \ \langle y^*,G(s)-G(s') \rangle \ge \frac{b}{4}-6c\ \ \ {\rm
or}\ \ \ \langle z^*,G(s)-G(s') \rangle \ge \frac{b}{4}-6c.$$ If $K$
was chosen big enough we then obtain that
$$\|G(s)-G(s')\|\ge \frac{K^{n-1}}{8} \ge \frac{\rho(s,s')}{L},$$
for some universal constant $L$.

\medskip {\bf f)} Finally assume that $n=m$ and $p=n-1$.

\noindent Let $v_{n-1}$ be the successor of $u_{n-1}$ so that
$v_{n-1}\le s_{n-1}$ and $w_{n-1}$ be the successor of $u_{n-1}$ so
that $w_{n-1}\le s'_{n-1}$. Denote now $x^*=x^*_{r_{n-1},n}$,
$y^*=x^*_{r_{n-2}\frown v_{n-1}}$ and $z^*=x^*_{r_{n-2}\frown
w_{n-1}}$. We also denote $|a|=|s_n|$ and
$b=|s_0\frown\dots\frown
s_{n-1}|-|u|=|s'_0\frown\dots\frown s'_{n-1}|-|u|$.

\noindent First, we have
$$\|G(s)-G(s')\|\ge \|G(s')-G(u)\|-3d \ge \alpha d'-3d,$$
where $\alpha \in (0,1)$ is a universal constant given by case (b).
If $d'\ge Md$, with $M=\displaystyle{\frac{6}{\alpha}}$, we obtain
that
$$\|G(s)-G(s')\|\ge \frac{\alpha}{2}d'\ge \frac{\alpha \rho(s,s')}{4}.$$
So, we may as well assume that $d'<Md$. Now, with our usual careful
choice of small $\delta_n$'s we get
$$\langle x^*,G(s)-G(s') \rangle \ge \frac{a}{4}-6b\ \ \ {\rm and}$$
$${\rm either} \ \ \langle y^*,G(s)-G(s') \rangle \ge \frac{b}{4}\ \ \ {\rm
or}\ \ \ \langle z^*,G(s)-G(s') \rangle \ge \frac{b}{4}.$$ Then,
using $x^*+25y^*$ or $x^*+25z^*$, we obtain that
$$\|G(s)-G(s')\|\ge \frac{d}{104} = \frac{(M+1)d}{104(M+1)}\ge
\frac{d+d'}{104(M+1)}=\frac{\rho(s,s')}{104(M+1)}.$$

\medskip All possible cases have been considered and our discussion is
finished.

\end{proof}

\section{On the non-embeddability of the hyperbolic trees}

Our aim is now to prove in the reflexive case the converse of the results given
in the previous section. More precisely, the main result of this section is the
following.

\begin{Thm}\label{converse} Assume that $X$ is a separable reflexive Banach space
and that there exists $C\ge 1$ such that $T_N \buildrel {C}\over
{\hookrightarrow} X$ for all $N$ in $\Ndb$.

\noindent Then either Sz$(X)>\omega$ or Sz$(X^*)>\omega$.
\end{Thm}

Before proceeding with the proof of this theorem, we need to recall two very
convenient renorming theorems essentially due to Odell and Schlumprecht. We
refer to \cite{OdellSchlumprecht2002} and \cite{OdellSchlumprecht2006} for a
complete exposition of the links between the Szlenk index of a Banach space and
its embeddability into a Banach space with a finite dimensional decomposition
with upper and lower estimates.

\begin{Thm}\label{odellschlumprecht0}Let $X$ be a separable reflexive Banach space.
Then, the following properties
are equivalent.

(i) Sz$(X)\le\omega$.

(ii) There exist $1<p<\infty$ and an equivalent norm $\|\cdot\|$ on $X$ such
that if $\mathcal U$ is a non-principal ultrafilter on $\mathbb N$, $x\in X$
and $(x_n)_{n=1}^{\infty}$ is any bounded sequence with $\lim_{n\in\mathcal
U}x_n=0$ weakly
\begin{equation}\label{pestimate}
\lim_{n\in\mathcal U}\|x+x_n\|\le \lim_{n\in\mathcal
U}(\|x\|^p+\|x_n\|^p)^{1/p}.
\end{equation} \end{Thm}

This is contained in the proof of Theorem 3 of \cite{OdellSchlumprecht2006}.

\begin{Thm}\label{odellschlumprecht}
Let $X$ be a separable reflexive Banach space. Then, the following properties
are equivalent.

(i) Sz$(X)\le\omega$ and Sz$(X^*)\le \omega$.

(ii) There exist $1<p<q<\infty$ and an equivalent norm $\|\cdot\|$ on $X$ such
that if $\mathcal U$ is a non-principal ultrafilter on $\mathbb N$, $x\in X$
and $(x_n)_{n=1}^{\infty}$ is any bounded sequence with $\lim_{n\in\mathcal
U}x_n=0$ weakly
\begin{equation}\label{pqestimate}
\lim_{n\in\mathcal U}(\|x\|^q+\|x_n\|^q)^{1/q}\le \lim_{n\in\mathcal
U}\|x+x_n\|\le \lim_{n\in\mathcal U}(\|x\|^p+\|x_n\|^p)^{1/p}.
\end{equation}
\end{Thm}

Let us remark that (ii) is equivalent to the statements that
$\overline{\delta}(\tau)\ge (1+\tau^q)^{1/q}-1$ and $\overline{\rho}(\tau)\le
(1+\tau^p)^{1/p}-1.$  This result follows directly from Theorem 7 of
\cite{OdellSchlumprecht2006}.

\begin{proof}[Proof of Theorem~\ref{converse}] Let $X$ be a reflexive Banach space such that Sz$(X)\le\omega$ and
Sz$(X^*)\le\omega$. We will assume that the norm satisfies \eqref{pqestimate}
and we may assume for convenience that $p$ and $q$ are conjugate i.e.
$\frac{1}{p}+\frac{1}{q}=1$.

Let us suppose that there is a constant $C\ge 1$ so that for every $N\in\mathbb
N$, we have $T_N \buildrel {C}\over {\hookrightarrow} X$.  We will show that
for large enough $N$ this produces a contradiction. Let us pick $a\in\mathbb N$
such that $a> (2C)^q.$  We then pick $m\in\mathbb N$ with $m>(2C)^q$ and
$N=a^{m+1}.$

Suppose now that $u:T_N \to X$ is a map such that $u(\emptyset)=0$ and:

\begin{equation}\label{embedTN}
\forall s,s'\in T_N\ \ \ \rho(s,s')\le \|u(s)-u(s')\| \le C\rho(s,s').
\end{equation}

We now consider an ultraproduct $\mathcal X$ of $X$ modeled on the
set $\mathbb N^N$; this idea is inspired by similar considerations
in \cite{MaureyMilmanTomczak1995}.  Let $\mathcal U$ be a fixed
non-principal ultrafilter on $\mathbb N$ and define the seminorm on
$Z=\ell_{\infty}(\mathbb N^N,X)$ by
$$ \|x\|_{\mathcal X}=\lim_{n_1\in\mathcal U}\cdots\lim_{n_N\in\mathcal U} \|x(n_1,\ldots,n_N)\|.$$
If we factor out the set $\{x:\|x\|_{\mathcal X}=0\}$ this induces an ultraproduct $\mathcal X.$

For $x\in Z$ and $0\le k\le N$ we define
$$ \mathcal E_k(x)(n_1,\ldots,n_N)=\lim_{n_{k+1}\in\mathcal U}\cdots\lim_{n_{N}\in\mathcal U}x(n_1,\ldots,n_N)$$
where each limit is with respect to the weak topology on $X$ (recall that $X$ is reflexive).
For $k<0$ it is convenient to write $\mathcal E_kx=0.$  It will be useful to
introduce $\mathcal F_k=I-\mathcal E_k$ for the complementary projections.

We now use \eqref{pqestimate} to deduce that if $\mathcal F_kx=0$ and $\mathcal
E_ky=0$ then
$$ (\|x\|_{\mathcal X}^q+\|y\|_{\mathcal X}^q)^{1/q}\le
\|x+y\|_{\mathcal X} \le (\|x\|_{\mathcal X}^p+\|y\|_{\mathcal
X}^p)^{1/p}.$$ From this it follows that the projections $\mathcal F_k$ are
contractive. Also if $0=k_0< k_1< k_2<k_r$ and $x_j\in Z$ with $\mathcal
F_{k_j}x_j=0$ and $\mathcal E_{k_{j-1}}x_j=0$ for $1\le j\le r$ then
\begin{equation}\label{pqestimates2}
\left(\sum_{j=1}^r\|x_j\|_{\mathcal X}^q\right)^{1/q}\le
\|\sum_{j=1}^rx_j\|_{\mathcal X}\le \left(\sum_{j=1}^r\|x_j\|_{\mathcal
X}^p\right)^{1/p}.\end{equation}

Let us now define $z_j\in Z$ for $1\le j\le N$ by
$$ z_j(n_1,\ldots,n_N)=u(n_1,\ldots,n_j)-u(n_1,\ldots,n_{j-1}).$$
Here we understand that $z_1(n_1,\ldots,n_N)=u(n_1).$

We then define $w_{j0}= z_j- \mathcal E_{j-1}z_j$ and then
$$ w_{jk} = \mathcal E_{j-a^{k-1}}z_j-\mathcal E_{j-a^k}z_j, \qquad 1\le k <\infty.$$
Then
$$ z_j=\sum_{k=0}^{\infty}w_{jk}$$ and by \eqref{pqestimates2}
\begin{align*}
\sum_{k=1}^m\|w_{jk}\|_{\mathcal X}&\le
m^{1/p}(\sum_{k=0}^\infty\|w_{jk}\|_{\mathcal X}^q)^{1/q}
\\ &\le m^{1/p}\|z_j\|_{\mathcal X}\le Cm^{1/p}.\end{align*}
This implies that
\begin{equation}\label{upper}
\sum_{j=1}^N\sum_{k=1}^m\|w_{jk}\|_{\mathcal X}\le
Cm^{1/p}N.\end{equation}

On the other hand if $0\le r\le r+s\le N$ we note that by \eqref{embedTN},
$$ \lim_{n'_{r+1}\in\mathcal U}\lim_{n'_{r+2}\in\mathcal U}\cdots\lim_{n'_{r+s}\in\mathcal U}
\|u(n_1,\ldots,n_r,n'_{r+1},\ldots n'_{r+s})-u(n_1,\ldots,n_{r+s})\| \ge 2s.$$
Hence if $v\in\ell_{\infty}(\mathbb N^r,X)$ we have
$$ \lim_{n_{r+1}\in\mathcal U}\cdots \lim_{n_{r+s}\in\mathcal U}\|u(n_1,\ldots,n_{r+s})-v(n_1,\ldots,n_r)\|\ge s.$$
In particular
if we let
$$ v(n_1,\ldots,n_r)=\lim_{n_{r+1}\in\mathcal U}\cdots\lim_{n_{r+s}\in\mathcal U}u(n_1,\ldots,n_{r+s})$$
(with limits in the weak topology) we obtain
$$ \|\mathcal F_r(\sum_{j=r+1}^{r+s}z_j)\|_{\mathcal X}\ge s.$$

Now suppose $s=a^k$ where $k\ge 1$.  If $r\le N-a^k$ we have
\begin{align*} a^k&\le \|\mathcal F_r(\sum_{j=r+1}^{r+a^k}z_j)\|_{\mathcal X}\\
&\le \|\sum_{j=r+1}^{r+a^k} \mathcal F_{j-a^k}z_j\|_{\mathcal X}.\end{align*}
The last inequality follows from the fact that $\mathcal F_k \mathcal
F_l=\mathcal F_l \mathcal F_k=\mathcal F_l$, whenever $k\le l$ and from the
contractivity of $\mathcal F_r$.

On the other hand
\begin{align*} \|\sum_{j=r+1}^{r+a^k} \mathcal F_{j-a^{k-1}}z_j\|_{\mathcal X} &=\|\sum_{j=r+1}^{r+a^{k-1}}\sum_{i=0}^{a-1}\mathcal F_{j+(i-1)a^{k-1}}z_{j+ia^{k-1}}\|_{\mathcal X}\\
&\le \sum_{j=r+1}^{r+a^{k-1}}\left(\sum_{i=0}^{a-1}\|\mathcal F_{j+(i-1)a^{k-1}}z_{j+ia^{k-1}}\|_{\mathcal X}^p\right)^{1/p}\\
&\le Ca^{k-1} a^{1/p}\\
&\le a^k/2.\end{align*} Combining these statements we have that if
$r=\lambda a^k$ with $1\le k\le m$ and $0\le \lambda \le a^{m+1-k}-1$ (in
particular $r\le N-a^k=a^{m+1}-a^k$)
$$ \sum_{j=r+1}^{r+a^k}\|w_{jk}\|_{\mathcal X}\ge a^k/2$$ and hence
$$ \sum_{j=1}^N \|w_{jk}\|_{\mathcal X}
=\sum_{\lambda=0}^{a^{m+1-k}-1}\sum_{j=\lambda a^k+1}^{\lambda a^k}
\|w_{jk}\|_{\mathcal X}   \ge \frac{N}{2}.$$ This implies
\begin{equation}\label{lower} \sum_{j=1}^N\sum_{k=1}^m\|w_{jk}\|_{\mathcal X}\ge \frac{mN}{2}.\end{equation}
Now \eqref{upper} and \eqref{lower} give a contradiction since
$m>(2C)^q.$
\end{proof}

As an immediate consequence of Theorem \ref{converse} and section
\ref{embeddings} we obtain the following characterization, which yields Theorem
\ref{main} announced in our introduction.

\begin{Cor} Let $X$ be a separable reflexive Banach space. The following assertions are
equivalent

(i) Sz$(X)>\omega$ or Sz$(X^*)>\omega$.

(ii) There exists $C\ge 1$ such that $T_\infty \buildrel {C}\over
{\hookrightarrow} X$.

(iii) There exists $C\ge 1$ such that for any $N$ in $\Ndb$, $T_N \buildrel
{C}\over {\hookrightarrow} X$.

\end{Cor}

\noindent {\bf Remark.} Let us mention that we do not know if (iii)
implies (i) for general Banach spaces.

\section{Applications to coarse Lipschitz embeddings and uniform homeomorphisms between Banach spaces}

We need to recall some definitions and notation. Let $(M,d)$ and $(N,\delta)$
be two unbounded metric spaces. We define
$$\forall t>0\ \ \ \omega_f(t)=\sup\{\delta(f(x),f(y)),\ x,y\in M,\ d(x,y)\le t\}.$$
We say that $f$ is {\it uniformly continuous} if $\lim_{t\to 0} \omega_f(t)=0$.
The map $f$ is said to be {\it coarsely continuous} if $\omega_f(t)<\infty$ for
some $t>0$.

\noindent Let us now introduce
$$L_\theta(f)= \sup_{t\ge \theta}\frac{\omega_f(t)}{t},\ \ {\rm for}\
\theta>0$$ and
$$L(f)=\sup_{\theta >0} L_\theta (f),\ \ \ \ L_\infty(f)=\inf_{\theta >0} L_\theta
(f).$$ A map is Lipschitz if and only if $L(f)<\infty$. We will say
that it is {\it coarse Lipschitz} if $L_\infty(f)<\infty$. Clearly,
a coarse Lipschitz map is coarsely continuous. If $f$ is bijective,
we will say that $f$ is a {\it uniform homeomorphism} (respectively,
{\it coarse homeomorphism, Lipschitz homeomorphism, coarse Lipschitz
homeomorphism}) if $f$ and $f^{-1}$ are uniformly continuous
(respectively, coarsely continuous, Lipschitz, coarse Lipschitz).
Finally we say that $f$ is a {\it coarse Lipschitz embedding} if it
is a coarse Lipschitz homeomorphism from $X$ onto $f(X)$.

We conclude this brief introduction with the following easy and well known
fact: if $X$ and $Y$ are Banach spaces, then for any map $f:X\to Y$, $\omega_f$
is a subadditive function. It follows that any coarsely continuous map $f: X\to
Y$ is coarse Lipschitz. In particular, any uniform homeomorphism is a coarse
Lipschitz homeomorphism.

\begin{Thm}\label{stability} Let $X$ and $Y$ be separable Banach spaces and suppose that there is
a coarse Lipschitz embedding of $X$ into $Y$.  Suppose $Y$ is reflexive and
$\Sz(Y)=\omega$.  Then $X$ is reflexive.
\end{Thm}

\begin{proof}
We can assume by Theorem \ref{odellschlumprecht0} that $Y$ is normed to satisfy
\eqref{pestimate} for some $1<p<\infty.$

Now let $f:X\to Y$ be a coarse Lipschitz embedding. We may assume that
there exists $C\ge 1$ such that
$$\|x_1-x_2\|-1\le \|f(x_1)-f(x_2)\|\le C\|x_1-x_2\|+1 \qquad x_1,x_2\in X.$$

Suppose that $X$ is a non reflexive Banach space and fix $\theta \in (0,1)$.
Then, James' Theorem \cite{James1964} insures the existence of a sequence
$(x_n)_n$ in $B_X$ such that $\|y-z\|\ge \theta$, for all $n\in \Ndb$, all $y$
in the convex hull of $\{x_i\}_{i=1}^n$ and all $z$ in the convex hull of
$\{x_i\}_{i\ge n+1}$. In particular
\begin{equation}\label{james}
\|x_{n_1}+..+x_{n_k}-(x_{m_1}+..+x_{m_k})\|\ge \theta k,\ \ \
n_1<..<n_k<m_1<..<m_k.
\end{equation}

For $k\in\mathbb N$ let $\mathbb N^{[k]}$ denote the collection of all
$k$-subsets of $\mathbb N$ (written in the form $(n_1,\ldots,n_k)$ where
$n_1<n_2<\cdots<n_k.$  We define $h:\mathbb N^{[k]}\to X$ by
$$ h(n_1,\ldots,n_k)= x_{n_1}+\cdots+x_{n_k}.$$  On $\mathbb N^{[k]}$ we define the distance
$$ d((n_1,\ldots,n_k),(m_1,\ldots,m_k))=|\{j:\ n_j\neq m_j\}|.$$
Then $h$ is Lipschitz with constant at most $2$.  Furthermore $f\circ h$ has
Lipschitz constant at most $2C+1$. By Theorem 4.2 of
\cite{KaltonRandrianarivony2008} there is an infinite subset $\mathbb M$ of
$\mathbb N$ so that $\diam f\circ h(\mathbb M^{[k]})\le 3(2C+1)k^{1/p}.$ If
$n_1<n_2<\cdots<n_k<m_1<\cdots<m_k\in\mathbb M$ we thus have
$$ \theta k-1\le \|f(x_{n_1}+\cdots+x_{n_k})-f(x_{m_1}+\cdots+x_{m_k})\|\le 3(2C+1)k^{1/p}.$$
For large enough $k$ this is a contradiction. \end{proof}

It is proved in \cite{GodefroyKaltonLancien2001} (Theorem 5.5) that the
condition ``having a Szlenk index equal to $\omega$" is stable under uniform
homeomorphisms. So we immediately deduce.

\begin{Cor}\label{uhstable} The class of all reflexive Banach spaces with Szlenk index equal to
$\omega$ is stable under uniform homeomorphisms.
\end{Cor}

As a final application we now state the main result of this section.

\begin{Thm}\label{coarsestable} Let $Y$ be a reflexive Banach space such that Sz$(Y)\le \omega$ and
Sz$(Y^*)\le \omega$ and assume that $X$ is a Banach space which coarse
Lipschitz embeds into $Y$. Then $X$ is reflexive, Sz$(X)\le \omega$ and
Sz$(X^*)\le \omega$.
\end{Thm}

\begin{proof} First, it follows from Theorem \ref{stability} that $X$ is reflexive.
Assume now that Sz$(X)$ or Sz$(X^*)$ is greater than $\omega$. Then, we know
from Theorem \ref{infinitetree} that $T_\infty$ Lipschitz embeds into $X$ and
therefore into $Y$. This is in contradiction with Theorem \ref{converse}.
\end{proof}

\begin{remark} Theorem \ref{stability}, Corollary \ref{uhstable} and
Theorem \ref{coarsestable} should be compared to the fact that in general
reflexivity is not preserved under coarse Lipschitz embeddings or even uniform
homeomorphisms. Indeed, Ribe proved in \cite{Ribe1984} that $\ell_1\oplus
(\sum_n\oplus \ell_{p_n})_{\ell_2}$ is uniformly homeomorphic to $(\sum_n\oplus
\ell_{p_n})_{\ell_2}$, if $(p_n)_n$ is strictly decreasing and tending to 1 (we
also refer to Theorem 10.28 in \cite{BenyaminiLindenstrauss2000} for a
generalization of this result). The space $X=(\sum_n\oplus
\ell_{p_n})_{\ell_2}$ is of course reflexive and standard computations yield
that its Szlenk index is equal to $\omega^2$. On the other hand, if the $p_n$'s
are chosen in $(1,2]$, it is also easy to show that the natural norm of $X^*$
is asymptotically uniformly smooth with a modulus of asymptotic smoothness
$\overline \rho(t)=t^2$. Thus, Sz$(X^*)=\omega$.

So, in view of Corollary \ref{uhstable} and Theorem \ref{coarsestable}, Ribe's
example is optimal.
\end{remark}

Let us now recall that for a separable Banach space the condition ``Sz$(X)\le
\omega$" is equivalent to the existence of an equivalent asymptotically
uniformly smooth norm on $X$ and that for a reflexive separable Banach space
the condition ``Sz$(X^*)\le \omega$" is equivalent to the existence of an
equivalent asymptotically uniformly convex norm on $X$ (see
\cite{OdellSchlumprecht2006} for a survey on these results and proper
references). Let us now denote as in \cite{OdellSchlumprecht2006}:
$$\cal C_{auc}=\{Y:\ Y\ { is\  separable\  reflexive\ and\ has\ an\ equivalent\
a.u.c.\ norm}\}$$ and
$$\cal C_{aus}=\{Y:\ Y\ { is\  separable\  reflexive\ and\ has\ an\ equivalent\
a.u.s.\ norm}\}.$$ Then, we can restate Corollary \ref{uhstable} and Theorem
\ref{coarsestable} as follows

\begin{Thm} The  class $\cal C_{aus}$ is stable under uniform homeomorphisms
and the class $\cal C_{auc}\cap \cal C_{aus}$ is stable under coarse Lipschitz
embeddings.
\end{Thm}

\end{document}